%% file: schedulingHardByEasyArxiv.tex
\documentclass[preprint]{elsarticle}

\makeatletter
\def\ps@pprintTitle{%
 \let\@oddhead\@empty
 \let\@evenhead\@empty
 \def\@oddfoot{}%
 \let\@evenfoot\@oddfoot}
\makeatother

\usepackage{graphicx, amsfonts,amsmath,amssymb,amsthm}
\usepackage{commands,enumitem}
\usetikzlibrary{patterns,decorations.pathreplacing}
\usepackage{hyperref}
\usepackage{natbib}
\usepackage{algorithm}
\usepackage{algcompatible}
\usepackage{rotating}

\newtheorem{theorem}{Theorem}
\newtheorem{proposition}[theorem]{Proposition}

\newdefinition{remark}{Remark}

\journal{European Journal of Operational Research}

\begin{document} 

\begin{frontmatter}

\title{
Learning to solve the single machine scheduling problem with release times and sum of completion times}


\author[mymainaddress]{Axel Parmentier}
\ead{axel.parmentier@enpc.fr}

\author[mysecondaryaddress]{Vincent T'Kindt\corref{mycorrespondingauthor}}
\cortext[mycorrespondingauthor]{Corresponding author}
\ead{tkindt@univ-tours.fr}

\address[mymainaddress]{CERMICS, Ecole des Ponts, Marne-la-Vallée, France}
\address[mysecondaryaddress]{University of Tours,\\ 
	LIFAT (EA 6300), ERL CNRS ROOT 7002, Tours, France}

\begin{abstract}
In this paper, we focus on the solution of a hard single machine scheduling problem by new heuristic algorithms embedding techniques from machine learning field and scheduling theory. These heuristics transform an instance of the hard problem into an instance of a simpler one solved to optimality. The obtained schedule is then transposed to the original problem. Computational experiments show that they are competitive with state-of-the-art heuristics, notably on large instances. 
\end{abstract}


\end{frontmatter}


\input{intro}
\input{scheme}
\input{ml}
\input{numeric}
\input{conclu}

\bibliography{biblio}

\appendix
\input{app_proofs}
\input{app_noise}

\end{document}

%% file: intro.tex

\section{Introduction} 
\label{sec:introduction}

Consider the problem where $n$ jobs have to be scheduled on a single machine. Each job $j$ is defined by a \emph{processing time} $p_j$ and a release date $r_j$ so that, in a give schedule, no job $j$ can start before its release date. The machine can only process one job at a time and preemption is not allowed. 
The goal is to find a schedule $s$ (permutation) that minimizes the total completion time $\sum_j C_j(s)$ with $C_j(s)$ the completion time of job $j$ in schedule $s$. If $s = (j_1,\ldots,j_n)$, then
$$C_{j_1}(s) = r_1 + p_1 \quad \text{and} \quad C_{j_k}(s) =  \max(C_{j_{k-1}}(s),r_{j_k}) + p_{j_k} \text{ for }k>1.
$$
When there is no ambiguity, we omit the reference to schedule $s$ when referring to completion times. Following the standard three-field notation in scheduling theory, this problem is referred to as 
$1|r_j|\sum_j C_j$ and is strongly $\cal NP$-hard \cite{AR76}.
When there is no release dates, the corresponding $1||\sum_j C_j$ problem can be solved in $O(n\log(n))$ time by means of the SPT rule (shortest processing times first). Again, when preemption is allowed, the corresponding $1|r_j, pmtn|\sum_j C_j$ problem can be solved $O(n\log(n))$ time by means of the SRPT rule (shortest remaining processing times first) \cite{KB74}.

The $1|r_j|\sum_j C_j$ problem is a challenging problem which has been studied for a long time. The two most competitive exact algorithms are the branch-and-memorize algorithm in \cite{STD20} and the dynamic programming based algorithm in \cite{TF12}, which are able to solve instances with a hundred of jobs. 
Heuristic algorithms can be used to compute good solutions in a reasonable amount of time. Along the years, numerous heuristic algorithms have been proposed. We cite the RDI local search based on the APRTF greedy rule proposed in \cite{CTU96} and which requires $O(n^4\log(n))$ time. The RBS heuristic (Recovering Beam Search) developed in \cite{DT02} is a truncated search tree approach that has been the most performing heuristic for a decade. The RBS heuristic requires $O(wn^3\log(n))$ time with $w$ the beam width parametrizing the heuristic: notice that in \cite{DT02} the case $w=1$ is considered. To the best of our knowledge, the state-of-the-art heuristic is a matheuristic proposed in \cite{DST14} which provides solutions very close to the optimal ones but at the price of a large CPU time requirement. All these heuristics have been milestones in the history of the $1|r_j|\sum_j C_j$ problem and will be considered as competitors for the learning based heuristics developed in this paper.

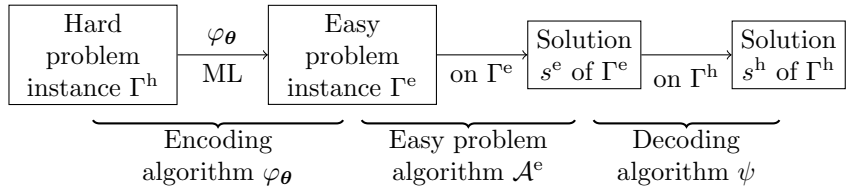
\begin{figure}[htp]
	\begin{center}
	
	\begin{tikzpicture}
	\node[draw, text width = 2cm,align=center] (a) at (0,0) {Hard problem \\ instance $\Gammah$};
	\node[draw, right=1.2cm of a.east, text width = 2cm, align=center] (b) {Easy problem \\ instance  $\Gammae$};
	\node[draw, right=1.2cm of b.east, align = center] (c) {Solution \\ $s^{\mathrm{e}}$ of $\Gammae$};
	\node[draw, right=1.2cm of c.east, align = center] (d) {Solution \\$s^{\mathrm{h}}$ of $\Gammah$};
	\draw[->] (a) to node[midway,above]{$\varphi_{\bftheta}$} node[midway,below]{ML} (b);
	\draw[->] (b) to node[midway,above]{} node[midway,below]{on $\Gammae$} (c);
	\draw[->] (c) to node[midway,above]{} node[midway,below]{on $\Gammah$} (d);
	
	\node[below= of a.south] (ba) {};
	\node (bb) at (a.south -| b) {};
	\node (bc) at (a.south -| c) {};
	\node (bd) at (a.south -| d) {};
	
	\draw [thick,decoration={brace, mirror, raise=0.2cm},decorate] (a.south) to node[midway,below=0.2cm, text width = 3cm, align = center] {Encoding \\ algorithm $\varphi_{\bftheta}$} (bb);
	\draw [thick,decoration={brace, mirror, raise=0.2cm},decorate] (bb) to node[midway,below=0.2cm, text width = 3cm, align=center] {Easy problem \\ algorithm $\calAe$} (bc);
	\draw [thick,decoration={brace, mirror, raise=0.2cm},decorate] (bc) to node[midway,below=0.2cm, text width = 3cm, align = center] {Decoding \\ algorithm $\psi$} (bd);
	\end{tikzpicture}
	\end{center}
	\caption{ML to approximate hard problems by well-solved ones}
	\label{fig:paradigm}
\end{figure}

The use of machine learning (ML) techniques within operations research (OR) algorithms is a recent but active and promising research area \citep{bengioMachineLearningCombinatorial2020}.
To the best of our knowledge, very few contributions of this kind have considered scheduling problems \citep{AOC16, FMN17, AS19}. 
In these works, ML is used to guide the solution process, {\it i.e.}, the proposed OR heuristic. In this paper, we elaborate on an original approach recently introduced in \citep{parmentierLearningApproximateIndustrial2019} and illustrated in Figure~\ref{fig:paradigm}. 
A ML predictor $\varphi_{\bftheta}$, which we call the encoding algorithm, is used to convert an instance $\Gammah$ of a {\it hard} optimization problem into an instance $\Gammae$ of an {\it easy} one, i.e., a problem for which a practically efficient algorithm $\calAe$ exists. 
Then, algorithm $\calAe$ is used to compute an optimal solution $s^{\mathrm{e}}$ of $\Gammae$.
And finally, a decoding algorithm is used to rebuild from $s^{\mathrm{e}}$ a solution $s^{\mathrm{h}}$ to $\Gammah$.
In this paper, the hard problem is the $1|r_j|\sum_j C_j$ problem and the easy problem is the $1||\sum_j C_j$ problem. The processing time $\hat p_j$ of job $j$ in $\Gammae$ is not equal to its processing time $p_j$ in $\Gammah$, but to a linear combination of features computed from $\Gammah$. 
The easy problem algorithm $\calAe$ is the SPT rule, and several different decoding algorithms are proposed to obtain a schedule for the $1|r_j|\sum_j C_j$ problem.

The main challenge to make such an approach work is to build an encoding algorithm $\varphi_{\bftheta}$ such that the optimal solution of the instance $\Gammae$ leads to a good solution $\Gammah$ after decoding.
As usual in ML, we first define an appropriate family of predictors $(\varphi_{\bftheta})_{\bftheta}$, and then seek (learn) the best parameter $\bftheta$.
Following \citep{parmentierLearningApproximateIndustrial2019}, we formulate the choice of $\bftheta$ as a structured learning problem.
Supervised learning is the branch of machine learning that is interested in constructing an approximation of an unknown function $f:x\mapsto y$ from a training set $(x_1,y_1),\ldots,(x_m,y_m)$, and structured learning \cite{nowozinStructuredLearningPrediction2010} is the specific branch of supervised learning where the output $y$ of $f$ belongs to a structured and combinatorially large set $\calY$.
In our case, $\calY$ is the set of schedules, i.e., the set of permutations of $[n]$. 
Structured learning algorithms on the permutation group have been thoroughly studied in the literature on rankings \citep{liuLearningRankInformation2011}, but with applications such as document retrieval, question answering or online advertising that are quite far from scheduling.
If the traditional learning approaches of this literature can in theory be applied in our context, we do not use them because the loss functions they use to evaluate if a ranking is a good approximation of another are not good evaluations of the quality of a $1|r_j|\sum_j C_j$ schedule.
We instead propose a novel approach based on Fenchel-Young \citep{blondelLearningFenchelYoungLosses} loss functions and inspired from the work of \citet{berthetLearningDifferentiablePerturbed2020}. This new approach is generic and can be applied to any hard problem for which the set of solutions of $\Gammae$ does not depend on $\varphi_{\bftheta}$.

The rest of the paper is organized as follows. Section \ref{sec:solution_scheme} introduces the learning based heuristics developped to solve the $1|r_j|\sum_j C_j$ problem. The proposed structured learning approach is presented in Section \ref{sec:structured_learning_methodology}. Extensive computational experiments showing the efficiency of the proposed heuristics are given in Section \ref{sec:numerical_results}. Finally, Section \ref{sec:conclusions} concludes the paper and provides further research directions.


%% file: scheme.tex

\section{Learning based heuristics for the $1|r_j|\sum_j C_j$ problem} 
\label{sec:solution_scheme}

To solve the $1|r_j|\sum_j C_j$ problem we propose three heuristic approaches.
The general scheme of these heuristics is the one we already described on Figure~\ref{fig:paradigm}: 
We first use a predictor $\varphi_{\bftheta}$ which transforms an instance  $\Gammah = \big(n,(p_j,r_j)_{j \in [n]}\big)$  into an instance $\Gammae = \big(n,(\hat p_j)_{j \in [n]}\big)$ of the $1||\sum_j C_j$ problem. Then, the SPT rule solves the latter in polynomial time. 
And finally, we decode from the solution of $\Gammae$ a solution of $\Gammah$.
Consequently, the role of the predictor is mainly to define processing times $\hat p_j$ for the $1||\sum_j C_j$ problem so that the schedule produced by the SPT rule is a good schedule when reintroducing the release dates $r_j$ and the original processing times $p_j$. Let  $s^{\mathrm{e}}$ be the schedule  given by the SPT rule on $\Gammae$.
Our three heuristics share the same predictor $\varphi_{\bftheta}$, which is described in Section~\ref{sub:encoding}.
The two first heuristics we propose are direct applications of the above principle, with different decoding algorithm and are developed in Section~\ref{sec:2firstH}. 
The third heuristic is more elaborated and is developed in Section~\ref{sec:3rdH}.

\subsection{Encoding algorithm}
\label{sub:encoding}
Let us now focus on how the predictor $\varphi_{\bftheta}$ is built. A natural way to define the mapping $\varphi_{\bftheta}$ that maps $\Gammah$ to an instance $\Gammae$ is through a features vector  $\bfphi : (j,\Gammah) \mapsto \bfphi(j;\Gammah)$ that associates to each job $j$ in $\Gammah$ a vector of features $\bfphi(j;\Gammah)$.
Each component of this vector of features provides information on job $j$. 
The choice of the features is a crucial point to make the approach efficient and it is discussed in Section \ref{sec:numerical_results}. We suppose $\bfphi(i;\Gammah)$ to be in $\bbR^d$, with $d$ the number of used features.
Given a vector $\bftheta \in \bbR^d$ of parameters, we define the instance $\Gammae = \varphi_{\bftheta}(\Gammah)$ as the instance of $1||\sum_j C_j$ with $n$ jobs whose processing times $\hat p_j$ are defined by
$$ \hat p_j = \langle \bftheta | \bfphi(j;\Gammah) \rangle \quad \forall j \in \{1,\ldots,n\}. $$
In other words, the easy problem on instance $\Gammah$ is
\begin{equation}\label{eq:easyProblem}
	\min_s \langle \bftheta | \bfphi(s;\Gammah) \rangle \quad \text{where} \quad \phi(s;\Gammah) = \sum_{i=1}^n (n-i+1) \bfphi(j_i;\Gammah) 
\end{equation}
and $s = (j_1,\ldots,j_n)$.
Note that depending on the choice of $\bftheta \in \bbR^d$, job processing times $\hat p_j$ can be negative, which is not a problem since the SPT rule still applies in that case.
The aim of the structured learning problem described in Section~\ref{sec:structured_learning_methodology} is to choose $\bftheta \in \bbR^d$ in such a way that the optimal solution of $\Gammae$ is a good solution for $\Gammah$.




\subsection{Decoding algorithm: Heuristics \texttt{PMLH} and \texttt{IMLH}}
\label{sec:2firstH}

Solutions of both $1|r_j|\sum_j C_j$ and $1||\sum_j C_j$ are schedules, i.e., permutations of $[n]$.
Hence, a solution of $1||\sum_j C_j$ can be seen as a solution of $1|r_j|\sum_j C_j$, albeit with different completion times due to the release dates.
The first heuristic we propose, denoted by \texttt{PMLH} ({\it Pure Machine Learning Heuristic}) and illustrated on Figure~\ref{fig:PMLH} therefore only returns the schedule $s^{\mathrm{e}}$ as a solution of $\Gammah$.
The decoding algorithms only determine job starting times obtained when taking the release dates $r_j$ and processing times $p_j$. It follows that heuristic \texttt{PMLH} requires $O(n\log(n))$ time.

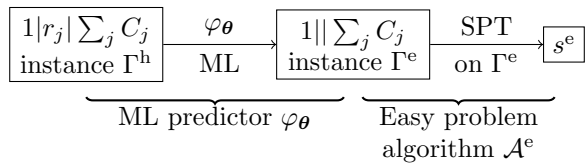
\begin{figure}[htp]
	\begin{center}
	
	\begin{tikzpicture}
	\node[draw, text width = 1.8cm,align=center] (a) at (0,0) {$1|r_j|\sum_j C_j$ instance $\Gammah$};
	\node[draw, right=1.5cm of a.east, text width = 1.8cm, align=center] (b) {$1||\sum_j C_j$ instance  $\Gammae$};
	\node[draw, right=1.5cm of b.east] (c) {$s^{\mathrm{e}}$};
	\draw[->] (a) to node[midway,above]{$\varphi_{\bftheta}$} node[midway,below]{ML} (b);
	\draw[->] (b) to node[midway,above]{SPT} node[midway,below]{on $\Gammae$} (c);
	
	\node[below= of a.south] (ba) {};
	\node (bb) at (a.south -| b) {};
	\node (bc) at (a.south -| c) {};
	
	\draw [thick,decoration={brace, mirror, raise=0.2cm},decorate] (a.south) to node[midway,below=0.2cm] {ML predictor $\varphi_{\bftheta}$} (bb);
	\draw [thick,decoration={brace, mirror, raise=0.2cm},decorate] (bb) to node[midway,below=0.2cm, text width = 3cm, align=center] {Easy problem \\ algorithm $\calAe$} (bc);
	\end{tikzpicture}
	\end{center}
	\caption{Heuristic \texttt{PMLH} general scheme.}
	\label{fig:PMLH}
	\end{figure}

It may happen that, in the schedule built from $s^{\mathrm{e}}$, two consecutive jobs $j$ and $k$, with $j\rightarrow k$, are such that $p_j>p_k$. Let $t$ be the starting time of $j$ in that schedule. Then, if $r_j,r_k\leq t$ the schedule is suboptimal and swapping $j$ and $k$ leads to a better schedule. In the second heuristic we propose, denoted by \texttt{IMLH} ({\it Improved Machine Learning Heuristic}) and illustrated on Figure~\ref{fig:solutionApproachOne}, we introduce in the decoding phase a fast local search \texttt{LS} to repair the above problematic cases, if any. This local search is given in Algorithm \ref{fig:solutionLS} and it has a $O(n^2)$ worst-case time complexity. To improve the schedule $s^i$ obtained from \texttt{LS}, we finally apply the \texttt{RDI} procedure described in \cite{CTU96}. More precisely, we apply the version of \texttt{RDI} where, each time a rescheduling of jobs is necessary, we use the SPT rule on the $\hat p_j$'s: always select among the available jobs at time $t$ the one with smallest $\hat p_j$. The \texttt{RDI} heuristic requires $O(n^4\log(n))$ time in the worst case. It follows that heuristic $\texttt{IMLH}$ requires $O(n^4\log(n))$ time in the worst case. 

\begin{figure}[htp]
\begin{center}

\begin{tikzpicture}
\node[draw, text width = 1.8cm,align=center] (a) at (0,0) {$1|r_j|\sum_j C_j$ instance $\Gammah$};
\node[draw, right=1.5cm of a.east, text width = 1.8cm, align=center] (b) {$1||\sum_j C_j$ instance  $\Gammae$};
\node[draw, right=1.5cm of b.east] (c) {$s^{\mathrm{e}}$};
\node[draw, right=1cm of c.east] (d) {$s^{\mathrm{i}}$};
\node[draw, right=1cm of d.east, align = center] (e) {solution \\$s$ of $\Gammah$};
\draw[->] (a) to node[midway,above]{$\varphi_{\bftheta}$} node[midway,below]{ML} (b);
\draw[->] (b) to node[midway,above]{SPT} node[midway,below]{on $\Gammae$} (c);
\draw[->] (c) to node[midway,above]{LS} node[midway,below]{on $\Gammah$} (d);
\draw[->] (d) to node[midway,above]{RDI} node[midway,below]{on $\Gammah$} (e);

\node[below= of a.south] (ba) {};
\node (bb) at (a.south -| b) {};
\node (bc) at (a.south -| c) {};
\node (be) at (a.south -| e) {};

\draw [thick,decoration={brace, mirror, raise=0.2cm},decorate] (a.south) to node[midway,below=0.2cm] {ML predictor $\varphi_{\bftheta}$} (bb);
\draw [thick,decoration={brace, mirror, raise=0.2cm},decorate] (bb) to node[midway,below=0.2cm, text width = 3cm, align=center] {Easy problem \\ algorithm $\calAe$} (bc);
\draw [thick,decoration={brace, mirror, raise=0.2cm},decorate] (bc) to node[midway,below=0.2cm] {Decoding algorithm} (be);
\end{tikzpicture}
\end{center}
\caption{Heuristic \texttt{IMLH} general scheme.}
\label{fig:solutionApproachOne}
\end{figure}
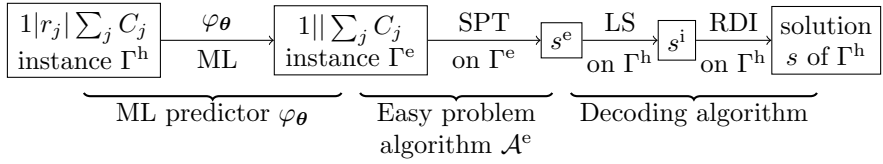

\begin{algorithm}[H]
	\begin{algorithmic}[1]
		\STATE \textbf{Input:} An instance $\Gammah$ and a schedule $s^{\mathrm{e}}$.
		\STATE \textbf{Output:} A schedule $s^i$. 
		\STATE $t=0$, $\ell=0$, $s^i=\emptyset$.
		\STATE $C_j=0, \forall j=1..n$. \hfill  \emph{$C_j$: job completion times in $s^i$}
		\WHILE{($\ell < (n-1)$)}
			\IF {($t<r_{s^{\mathrm{e}}[\ell]}$)} \hfill \emph{$s^{\mathrm{e}}[\ell]$ refers to the job in position $\ell$ in $s^{\mathrm{e}}$}
				\STATE $t=r_{s^{\mathrm{e}}[\ell]}$.
			\ENDIF
			\IF {($t\geq r_{s^{\mathrm{e}}[\ell+1]}$ and $p_{s^{\mathrm{e}}[\ell]}>p_{s^{\mathrm{e}}[\ell+1]})$} 
				\STATE Swap $s^{\mathrm{e}}[\ell]$ and $s^{\mathrm{e}}[\ell+1]$.
				\IF {($\ell=0$)}
					\STATE $t=0$.
				\ENDIF
				\IF {($\ell=1$)}
					\STATE $t=0$, $\ell=\ell-1$.
				\ENDIF
				\IF {($\ell>1$)}
					\STATE $t=C_{s^i[\ell-2]}$, $\ell=\ell-1$.
				\ENDIF
			\ELSE
			\STATE $t=t+p_{s^{\mathrm{e}}[\ell]}$, $s^{i}[\ell]=s^{\mathrm{e}}[\ell]$, $C_{s^i[\ell]}=t$, $\ell=\ell+1$.
			\ENDIF
		\ENDWHILE 
		\STATE {\bf return} $s^i$.
	\end{algorithmic}
	\caption{The LS repairing heuristic}
	\label{fig:solutionLS}
\end{algorithm}

\subsection{Decoding algorithm: Heuristic \texttt{itMLH}}
\label{sec:3rdH}

Heuristic \texttt{itMLH} (iterative MLH) is an extension of heuristic \texttt{IMLH} since $m$ vectors $\bftheta_k \in \bbR^d$ are generated by applying perturbations on the vector $\bftheta$ used in \texttt{IMLH}. The global scheme of \texttt{itMLH} is given in Figure~\ref{fig:severalApproximationsApproach}. 
In this approach, we choose to approximate $\Gammah$ by $m$ instances of the easy problem $\Gammae_1,\ldots,\Gammae_m$.
We suggest to use the instances:
\begin{equation}\label{eq:mInstancesGammaEDefinition}
\Gammae_i = \varphi_{\bftheta_k}(\Gammah)
\end{equation}
with $\bftheta_k = \bftheta + z_k$ and $\bfz_k$ is a sample of a random variable $\bfZ$ on $\bbR^d$, that can be chosen arbitrarily. The processing time $\hat p_{j,k}$ of job $j$ in instance $\Gammae_k$ is thus:
$$\hat p_{j,k} = \langle \bftheta_k | \bfphi(j;\Gammah)\rangle. $$
Note that if the results of \texttt{PMLH} and \texttt{IMLH} depend on the direction of $\bftheta$ but not on its norm, the result of \texttt{itMLH} depends on its direction and its norm because of perturbation $\bfZ$.
When learning the parameter $\bftheta$ for \texttt{itMLH}, we must therefore take into account the distribution of the perturbation $\bfZ$ that is used in \texttt{itMLH}.
This is what we do in Section~\ref{sec:structured_learning_methodology}. 
In our numerical experiments, we use a $\bfZ$ distributed according to the normal distribution $\calN(0,I_d)$ where $I_d$ is the identity matrix of dimension~$d$.

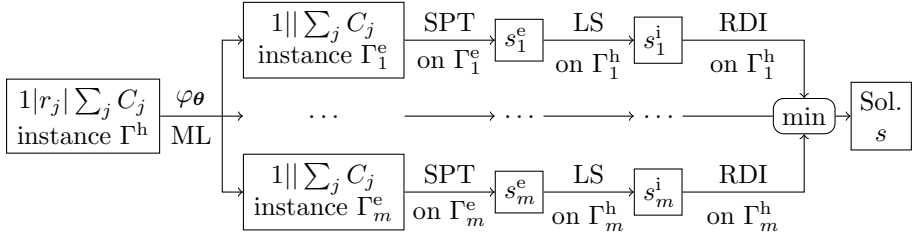
\begin{figure}
\begin{tikzpicture}
\node[draw, text width = 1.8cm,align=center] (a) at (0,0) {$1|r_j|\sum_j C_j$ instance $\Gammah$};
\node[draw, above right=0.5cm and 1.1cm of a.east, text width = 1.9cm, align=center] (b1) {$1||\sum_j C_j$ instance  $\Gammae_1$};
\node[right=1.1cm of a.east, text width = 1.9cm, align=center] (b2) {$\ldots$};
\node[draw, below right=0.5cm and 1.1cm of a.east, text width = 1.9cm, align=center] (b3) {$1||\sum_j C_j$ instance  $\Gammae_m$};
\node[draw, right=1.2cm of b1.east, text width=0.4cm] (c1) {$s_{1}^{\mathrm{e}}$};
\node[right=1.2cm of b2.east, text width=0.4cm] (c2) {$\ldots$};
\node[draw, right=1.2cm of b3.east, text width=0.4cm] (c3) {$s_{m}^{\mathrm{e}}$};
\node[draw, right=1.2cm of c1.east, text width=0.4cm] (d1) {$s_{1}^{\mathrm{i}}$};
\node[right=1.2cm of c2.east, text width=0.4cm] (d2) {$\ldots$};
\node[draw, right=1.2cm of c3.east, text width=0.4cm] (d3) {$s_{m}^{\mathrm{i}}$};

\node[left=0.15cm of b1.west] (bm1) {};
\node[left=0.15cm of b2.west] (bm2) {};
\node[left=0.15cm of b3.west] (bm3) {};
\draw[<->] (b1.west) -- (bm1.center) -- (bm3.center) -- (b3.west);

\node[draw, rounded corners, right=1.2cm of d2.east] (dp2) {$\min$};
\node (dp1) at (d1.east -| dp2) {};
\node (dp3) at (d3.east -| dp2) {};
\draw[->] (d1) to node[midway, above] {RDI} node[midway,below] {on $\Gammah_1$} (dp1.center) to (dp2);
\draw[->] (d3) to node[midway, above] {RDI} node[midway,below] {on $\Gammah_m$} (dp3.center) to (dp2);

\node[draw, right=0.2cm of dp2.east, align = center] (e) {Sol.\\$s$};

\draw[->] (a) to node[midway,above]{$\varphi_{\bftheta}$} node[midway,below]{ML} (bm2.center) -- (b2);
\draw[->] (b1) to node[midway,above]{SPT} node[midway,below]{on $\Gammae_1$} (c1);
 \draw[->] (b2) to (c2);
\draw[->] (b3) to node[midway,above]{SPT} node[midway,below]{on $\Gammae_m$} (c3);
\draw[->] (c1) to node[midway,above]{LS} node[midway,below]{on $\Gammah_1$} (d1);
 \draw[->] (c2) to (d2);
\draw[->] (c3) to node[midway,above]{LS} node[midway,below]{on $\Gammah_m$} (d3);
\draw[->] (d2) -- (dp2) to node[midway,above]{} node[midway,below]{} (e);
\end{tikzpicture}
\caption{Solution approach using several approximations}
\label{fig:severalApproximationsApproach}
\end{figure}

The heuristic \texttt{itMLH} starts by running $\texttt{IMLH}$ before generating iteratively the $m$ instances $\Gammae_k$. To speed up the heuristic, at each iteration $k$, if the schedule $s^{\mathrm{e}_k}$ has already been obtained at a previous iteration $k'<k$, then neither \texttt{LS} nor \texttt{RDI} are applied and the heuristic starts iteration ($k+1$). Similarly, \texttt{RDI} is not applied whenever $s^i_k$ has already been computed at a previous iteration. It follows that heuristic \texttt{itMLH} requires $O(mn^4\log(n))$ time in the worst case.
Finally, remark that \texttt{itMLH} can be easily parallelized.





%% file: ml.tex
\section{Structured learning methodology}
\label{sec:structured_learning_methodology}

\subsection{Background on structured learning}
\label{sub:background_on_structured_learning}

After a brief introduction to structured learning which follows \citep{parmentierLearningApproximateIndustrial2019}, we introduce our own variant of the perturbed optimizers and the resulting Fenchel-Young loss structured learning problem of \citet{berthetLearningDifferentiablePerturbed2020}. 

\subsubsection{Setting} 
\label{ssub:setting}

Broadly speaking, supervised learning  \citep{nowozinStructuredLearningPrediction2010} aims at learning an unknown function
\begin{equation*}
	\begin{array}{rcl}
		f:\calX &  \rightarrow & \calY \\
		x & \mapsto & y
	\end{array}
\end{equation*}
from a training set $(x_1,y_1),\ldots,(x_n,y_n)$ where $y_i$ is a noisy observation of $f(x_i)$.
Structured learning \citep{nowozinStructuredLearningPrediction2010} is a branch of supervised learning which
deals with problems such that, for any $x$ in $\calX$, the mapping $f(x)$ takes its value in a set $\calY(x)$ that is finite, combinatorially large, and structured. 
To predict the value $f(x)$ on a new instance $x$, we solve the following optimization problem:
\begin{equation}\label{eq:structuredPredictionProblem}
	y^*= \hat f_{\bftheta} (x) := \argmax_{y \in \calY(x)} g_{\bftheta}(y,x)
\end{equation}
where the statistical model $\hat f_{\bftheta}$ is defined through an auxiliary \emph{evaluation function} $g_{\bftheta}$.
Problem~\eqref{eq:structuredPredictionProblem} is called the \emph{structured prediction problem}.
Given a training set $(x_i,y_i)_{i \in [n]}$, the objective of the \emph{structured learning problem} is to learn a parameter ${\bftheta}$ such that $\hat f_{\bftheta}$ is a good approximation of $f$.
Several approaches have been considered in the structured learning literature to formulate the structured learning problem, notably based on the maximum likelihood estimators (MLE) or on the minimization of surrogate loss functions \citep{nowozinStructuredLearningPrediction2010}. 
Contrarily to \citep{parmentierLearningApproximateIndustrial2019}, we do not use a MLE approach, because computing the likelihood would require to evaluate a sum on all the elements of the permutation group, which is a difficult problem that would require an ad-hoc algorithm.
We instead introduce a Fenchel-Young loss approach, which leads to a much easier learning problem. 


In the rest of the paper, we assume to have a \emph{feature map} $\bfphi : (x,y) \mapsto \bfphi(y;x)$ which associates to every $y$ in $\calY(x)$ a vector $\bfphi(y;x)$ in $\bbR^d$ of features describing the properties of $y$. 
And we restrict ourselves to linear models of the form:
$$g_{\bftheta}(y,x) := \langle \bftheta | \bfphi(y;x)\rangle, $$ 
which leads to structured prediction problems of the form
\begin{equation}\label{eq:linearStructuredPredictionProblem}
	\max_{y \in \calY(x)}\langle \bftheta | \bfphi(y;x)\rangle.
\end{equation}

Let $\calC(x)$ be the convex hull of $\{\bfphi(y;x)\colon y \in \calY(x)\}$ and $\calS(x)$ be its linear span.
Remark that it is useless to consider features $\bfphi(y;x)$ that are affinely dependent for all $y \in \calY$. 
Indeed, if $\calS(x)\subsetneq \bbR^d$, and we decompose $\bftheta$ into $\bftheta_1 + \bftheta_2$ where $\bftheta_1$ is the orthogonal projection of $\bftheta$ on $\calS(x)$ and $\bftheta_2$ on its orthogonal complement, then the value of $\bftheta_2$ does not impact the result of the structured prediction problem.
And we can obtain the same prediction flexibility by removing some features so as to obtain a full dimensional $\calC(x)$. This justifies that, without loss of generality, we assume in the rest of the paper that $\calC(x)$ has a non-empty interior (for at least some $x$).
We denote by $\bfeta$ the vectors in $\calS(x)$. 

\subsubsection{Perturbed prediction problem} 
\label{ssub:perturbed_prediction_problem}

As described in Section~\ref{sec:3rdH}, we add to $\bftheta$ a random perturbation $\bfZ$, with $\bfZ$ a random variable on $\bbR^d$.  
Let be:
\begin{align*}	
	F(\bftheta;x) &:= \bbE_{\bfZ} \Big( \max_{y \in \calY(x)} \langle \bftheta + \bfZ| \bfphi(y;x)\rangle \Big) = \bbE_{\bfZ} \Big( \max_{\bfeta \in \calC(x)} \langle \bftheta + \bfZ| \bfeta\rangle \Big) 
	\quad \text{and,} \\ 
	\bfphi^*(\bftheta;\bfx) &:= \bbE_{\bfZ}\Big(\bfphi\Big(\argmax_{y \in \calY(x)}\big\langle\bftheta + \bfZ | \bfphi(y;x)\big\rangle;x\Big)\Big) = \bbE_{\bfZ}\Big(\argmax_{\bfeta \in \calC(x)}\langle \bftheta + \bfZ | \bfeta \rangle\Big).
\end{align*}
We denote by $\Omega(\cdot;x)$ the Fenchel dual of $F(\cdot;x)$, {\it i.e.}, 
$$\Omega(\bfeta;x) := \argmax_{\bftheta \in \calC(x)}\langle \bfeta | \bftheta \rangle - F(\bfeta;y).$$
\noindent The proof of the following result is available in~\ref{sec:proofs}.
\begin{proposition}\label{prop:propertiesOfF}
	The mapping $F(\cdot;x) : \bftheta \mapsto F(\bftheta;x)$ is convex with subgradient $\bfphi^*(\bftheta;\bfx)$, which by abuse of notations we denote by $\nabla_{\bftheta} F(\bftheta;x)$.
	\begin{enumerate}
		\item If $Z$ has positive and differentiable density $\rmd\mu(z) \propto \exp(-\nu(z))\rmd z$, then $F(\cdot;x)$ is twice differentiable and the subgradient above is its gradient. If in addition $\calC(x)$ has non-empty interior, then $F(\cdot;x)$ is strictly convex.
		\item If $Z$ has a sampled distribution $\frac{1}{m}\sum_{i=1}^m\delta_{\bfz_i}$, then $F(\cdot;x)$ is piecewise linear. 
	\end{enumerate}
\end{proposition}

\subsubsection{Perturbed model learning with Fenchel-Young losses} 
\label{ssub:perturbed_model_learning_with_fenchel_young_losses}

In the loss approach to supervised learning, we suppose to have a loss $\ell(\bftheta;y,x)$ that evaluates how far the $y$ predicted by~\eqref{eq:structuredPredictionProblem} is from the true $y$, and we formulate the learning problem as
$$ \min_{\bftheta}\frac{1}{n}\sum_{i=1}^n \ell(\bftheta;y_i,x_i).$$
An ideal loss $\ell(\bftheta;y,x)$ for the structured learning problem would be a loss that is non-negative, easy to minimize and such that $\ell(\bftheta;y,x)=0$ if and only if the $y^*$ realizing the maximum in \eqref{eq:linearStructuredPredictionProblem} is such that $y^* = y$.
But since predictions are done using \eqref{eq:linearStructuredPredictionProblem} where $y$ appears only through $\bfphi(y;x)$, we cannot hope to distinguish $y_1$ and $y_2$ if $\bfphi(y_1;x) = \bfphi(y_2;x)$.
Thus, all we can ask for is that $\ell(\bftheta;y,x) = 0$ if and only if $\bfphi(y;x) \in \argmax_{\bfeta \in \calC(x)}\langle \bftheta | \bfeta \rangle$.\\

Given $x$ and $y \in \calY(x)$, the \emph{Fenchel-Young loss} $L(\bftheta;y,x)$ is defined by:
$$ L(\bftheta;y,x) = F(\bftheta;x) + \Omega(\bfphi(y;x);x) - \langle \bftheta | \bfphi(y;x)\rangle. $$
It follows from Proposition~\ref{prop:propertiesOfF} that the following vector, which by abuse of notations we denote by $\nabla_{\bftheta}L(\bftheta;y,x)$, is a subgradient of $L(\cdot;y,x)$.
	\begin{equation}\label{eq:gradientLoss}
		\nabla_{\bftheta}L(\bftheta;y,x) = \bfphi^*(\bftheta;x) - \bfphi(y;x)
	\end{equation}
Given $\bfz$ in $\bbR^d$, we denote by $\delta_{\bfz}$ the Dirac distribution in $\bfz$.
The proof of the following result is available in \ref{sec:proofs}.
\begin{proposition}\label{prop:propertiesOfFYL}
	The Fenchel-Young loss $L(\cdot;y,x)\colon\bftheta \mapsto L(\bftheta;y,x)$ is non-negative, convex in $\bftheta$, and with value $0$ if and only if $\bfphi(y;x) \in \partial_{\bftheta}F(\bftheta;x)$, where $\partial_{\bftheta}F(\cdot;x)$ is the subdifferential of $F(\cdot;x)$.
	\begin{enumerate}
		\item If $Z$ has positive and differentiable density $\rmd\mu(z) \propto \exp(-\nu(z))\rmd z$ and $\calC(x)$ has non-empty interior, then $L(\cdot;y,x)$ is strictly convex and $L(\bftheta;y,x) = 0$ if and only if $\bfphi(y;x) = \bfphi^*(\bftheta;x)$.
		\item If $Z$ has a sampled distribution $\frac{1}{m}\sum_{i=1}^m\delta_{\bfz_i}$, then $L(\bftheta;y,x) = 0$ if and only if there exists $\displaystyle\bfeta_1 \in \argmax_{\bfeta \in \calC(x)}\langle \bftheta + \bfz_1|\bfeta\rangle,\ldots, \bfeta_m \in \argmax_{\bfeta \in \calC(x)}\langle \bftheta + \bfz_m|\bfeta\rangle$ such that $\bfphi(y;x) = \frac{1}{m}\sum_{i=1}^m \bfeta_i$.
	\end{enumerate}
\end{proposition}


Given a training set $(x_1,y_1),\ldots,(x_n,y_n)$, we therefore formulate the learning problem as
\begin{equation}\label{eq:learningProblem}
	\min_{\bftheta\in \bbR^d} \frac{1}{n} \sum_{i=1}^n L(\bftheta;y_i,x_i).
\end{equation}
In practice, we use a normal distribution $\calN(0,I_d)$ for $\bfZ$, and the first point of Proposition~\eqref{prop:propertiesOfFYL} ensures that~\eqref{eq:learningProblem} is a strictly convex optimization problem provided that $\calC(x_i)$ has non-empty interior for at least one $i$, which is practically the case as we explained at the end of Section~\ref{ssub:setting}.

The practical difficulty of the problem~\eqref{eq:learningProblem} lies in the integrals that appear in $F$ and its gradient, even if we omit the hard to compute constant terms $\Omega(y_i)$ in the objective function. Therefore, we replace it by its \emph{sample average approximation} (SAA): we draw $m$ samples $\bfz_1,\ldots,\bfz_m$ of $Z$ and replace $\mu$ by its sample average approximation $\tilde{\mu}$.
The consequence is to take the expectations with respect to $\tilde{\mu}$ when computing $F$ in the objective function and $\bfphi_\varepsilon^*$ in its gradient \eqref{eq:gradientLoss}.
The second point of Proposition~\ref{prop:propertiesOfFYL} then enables to interpret the $\bftheta$ obtained.
And the second point of Proposition~\ref{prop:propertiesOfF} ensures that the SAA problem is a convex and piecewise-linear non-constrained optimization problem.
Given that ~\eqref{eq:gradientLoss} ensures that subgradients can be computed by solving several problems of the form~\eqref{eq:linearStructuredPredictionProblem}, we can therefore solve the SAA problem using a subgradient algorithm, or any other non-differentiable convex optimization method.
\citet{lewisNonsmoothOptimizationQuasiNewton2013} and others have observed that BFGS algorithms perform well in practice on many such problems, although there is no theoretical convergence guarantee.
We therefore solve the problem with a BFGS algorithm. 
And we have observed a fast convergence to an optimal solution -- the certificate being a null gradient.
Altenative approaches could of course be used.

\begin{remark}
	The reader familiar with Machine Learning will note that we do not add a regularization term in~\eqref{eq:learningProblem}. 
	Indeed, the perturbation $\bfZ$ can be seen as a regularizer on the features space, as it is generally the case with Fenchel Young losses (see~\citet{blondelLearningFenchelYoungLosses} for more details).
\end{remark}

\begin{remark}
	\citet{berthetLearningDifferentiablePerturbed2020} solve the non-sampled learning problem using a stochastic gradient descent, which is required in their case because their $\bfphi(y;x)$ is the output of a deep neural network that they fit using this algorithm. 
	As we do not have this practical constraint, we have chosen to minimize the SAA approximation with a BFGS algorithm because it is easy to implement and gives a fast convergence in practice.
\end{remark}


\subsection{Learning to approximate the $1|r_j|\sum_j C_j$ problem} 
\label{sub:learning_to_approximate_}

We formulate the problem of learning the parameter $\bftheta$ of our scheme of Section~\ref{sec:solution_scheme} as a structured learning problem. 
In this case, the function $f$ we approximate is the function:
\begin{equation*}
\begin{array}{rcl}
	f : \Gamma \mapsto s^* (\Gamma) \quad \text{where $s^*(\Gamma)$ is an optimal solution $\Gamma$},
\end{array}
\end{equation*}
and $\Gamma$ is an instance of the $1|r_j|\sum_j C_j$ problem.
The structured prediction problem~\eqref{eq:linearStructuredPredictionProblem} then coincides with the easy problem~\eqref{eq:easyProblem}, the only difference being that the maximum is replaced by a minimum.
All we have to do is therefore to build a training set $(\Gamma_1,s_1), \ldots, (\Gamma_n,s_n)$ of instances of the scheduling problem with their optimal solution and solve the learning problem of Section~\ref{ssub:perturbed_model_learning_with_fenchel_young_losses} to obtain a parameter $\bftheta^*$.
We can then use $-\bftheta^*$ as parameter $\bftheta$ in~\eqref{eq:easyProblem}.

\begin{remark}
	In their work on perturbed maximizers, \citet{berthetLearningDifferentiablePerturbed2020} perturb the feature vector $\bfphi$ instead of the parameter $\bftheta$.
	For a reader that would compare the approaches, we underline that our $\bfphi$ correspond to their $\bftheta$ (see notably Section~5.2 in \cite{berthetLearningDifferentiablePerturbed2020}).
	We have chosen to use this alternative perturbation because it leads to better numerical results on the problem we consider.
\end{remark}

\subsection{Broad applicability}

In this paper, we use a structured learning approach to approximate instances of the hard problem $1|r_j|\sum_j C_j$  by easier instances of the $1||\sum_j C_j$ problem.
As detailed in \citep{parmentierLearningApproximateIndustrial2019}, the paradigm of using machine learning to approximate instances of a hard problem by instances of an easier one is quite generic, and can be used, for instance, to linearize problems.
However, the structured learning algorithm used to learn $\bftheta$ in~\citep{parmentierLearningApproximateIndustrial2019} is not generic: ad-hoc algorithms must be derived for new applications.
A strong advantage of the approach proposed in paper is that it is completely generic.
Indeed, the only things we need to apply it to new problems is:
\begin{enumerate}
	\item a training set composed of hard problem instances with their optimal solutions (or good quality solutions),
	\item a vector of features $\bfphi(y;x)$ -- See Section 1.3 of~\citep{parmentierLearningApproximateIndustrial2019} for details on how such vector of features can be built,
	\item an algorithm solving the easy problem.
\end{enumerate}
Indeed, an algorithm solving instances of the easy problem is all we need to compute $F(\bftheta;x)$ and its gradient.
And such an algorithm is always available: the existence of such an algorithm is the reason why we want to approximate the hard problem by this specific easy problem.
The structured learning approach based on Fenchel-Young losses detailled in this paper is therefore much more generic than the MLE approach, because it does not require additional algorithms. 





%% file: numeric.tex

\section{Computational experiments} 
\label{sec:numerical_results}

In this section we focus on the experiments done to build the predictor $\varphi_{\bftheta}$ and the comparisons of the learning based heuristics with respect to state-of-the-art algorithms. We first introduce how instances $\Gammah$ of the $1|r_j|\sum_j C_j$ problem are randomly generated (\cite{DT02}). For a given instance of $n$ jobs, processing times $p_j$ are drawn at random following the uniform distribution $[1;100]$ and release times $r_j$ are drawn at random following the uniform distribution $[1;50.5*n*\rho]$. Parameter $\rho$ enables to generate instances of different difficulties: we consider $\rho\in\{0.2, 0.4, 0.6, 0.8, 1.0, 1.25, 1.5, 1.75, 2.0, 3.0\}$. For each values of $n$ and $\rho$, $N$ instances are randomly generated leading for a fixed value of $n$ to $10*N$ instances. The values considered for $n$ and $N$ depend on the conducted experiments and are given in the next sections.

The algorithms considered in the experiments are the learning based heuristics \texttt{PMLH}, \texttt{IMLH} and \texttt{itMLH} with $m=150$. 
To evaluate the contribution of the predictor, we also implemented a version of \texttt{IMLH} in which instead of using the predictor to build an initial solution, we randomly generate it. Next, \texttt{LS} local search is applied to repair local suboptimal jobs sequencing. This heuristic is denoted by \texttt{RAND}.\\ 
We also use the RDI/APRTF heuristic (\cite{CTU96}), referred to as \texttt{RDIA}, the RBS heuristic with beam width $w=2$ (\cite{DT02}), referred to as \texttt{RBS}, and the matheuristic (\cite{DST14}), referred to as \texttt{MATH}. Besides, for instances with up to $n=110$ jobs, the optimal solution is computed by means of the branch-and-memorize algorithm in \cite{STD20}. 
All heuristic algorithms but \texttt{MATH} have been coded in C++ language. The code of heuristic \texttt{MATH} has been kindly provided by Fabio Salassa (\cite{DST14}). Testings have been done on a PC Intel XEON E5 with 8 cores of 2.20Ghz and 8Gb of RAM. 

\subsection{Features definition and selection}
\label{sec:LtP}

To build a predictor $\varphi_{\bftheta}$, the vector of features $\bfphi(j;\Gammah)$ must be defined.
We initially considered $66$ features,
and manually selected $d=27$ of these features that lead to good prediction performances.
This selection was done using a validation set distinct from the test set later used.
The $\bfphi(j;\Gammah)$ vector of (normalized) features is given in Table \ref{tab:TabFeat}.
We denote by $[j]^X$ the position of job $j$ in the sequence obtained by applying the sequencing rule $X\in\{$SPT, SRT, SP+RT$\}$. Rule SRT (resp. SP+RT) consists in sorting the jobs by increase value of the $r_j$'s (resp. $r_j+p_j$'s). Besides, features 15-17 and  23-27 are computed exploiting the SRPT algorithm (\cite{KB74}) which solves the preemptive version of the $1|r_j|\sum_j C_j$ problem, thus enabling to get in polynomial time a lower bound to the original problem. In a preemptive schedule computed by SRPT, some jobs $j$ can be interrupted at time $r_k$, while being processed, by the arrival of a job $k$ such that $p_k<(p_j-\pi_j)$ with $\pi_j$ the duration $j$ has been processed before its first preemption by job $k$. In case job $j$ is not preempted, we have $\pi_j=p_j$. Let $\#_j$ be the number of times job $j$ is preempted by another job, and $\#_T$ be the total number of preemptions. Finally, $[j]^{SRPT}$ is the position of the completion time of job $j$ in the SRPT schedule. For feature 16, $p_k$ refers to the processing time of the job that follows the initial part of job $j$ and of duration $\pi_j$. For features 18-21, $Decile(X)$ refers to the the decile number of $X\in\{r_j,p_j\}$ when the $X$'s are sorted by increasing value. For features 24-27, set $BS^p_j=\{k~|~[k]^{SRPT}<[j]^{SRPT}, p_k<p_j\}$, set  $BS^r_j=\{k~|~[k]^{SRPT}<[j]^{SRPT}, r_k<r_j\}$, set $BG^p_j=\{k~|~[k]^{SRPT}<[j]^{SRPT}, p_k>p_j\}$  and set $BG^r_j=\{k~|~[k]^{SRPT}<[j]^{SRPT}, r_k>r_j\}$.

\begin{table}[htp]
\begin{center}
	{\tiny
	\begin{tabular}{|c|c|r|r||c|c|r|r|}
	\hline
	 & Description & \multicolumn{1}{c|}{$\theta_k$}  & \multicolumn{1}{c||}{$\sigma_k$} && Description & \multicolumn{1}{c|}{$\theta_k$} & \multicolumn{1}{c||}{$\sigma_k$}\\ [0.2cm]
	 \hline
	 1 & $\frac{[j]^{SPT}}{n}$ &  4.05111 & 0.28865
	 & 15 & $\frac{p_j-\pi_j}{\sum_{i=1}^n p_i-\pi_i}$& -38.67500 & 0.03549\\ [0.2cm]
	 \hline
	 2 & $\frac{[j]^{SRT}}{n}$ &  -11.38040 & 0.28865
	 & 16 & $\frac{p_j-\pi_j}{p_k\sum_{i=1}^n p_i-\pi_i}$& 8.72219 & 0.04948\\ [0.2cm]
	 \hline
	 3 & $\frac{[j]^{SP+RT}}{n}$ &  -10.19020& 0.28865
	 & 17 & $\frac{p_j-\pi_j}{p_j\sum_{i=1}^n p_i-\pi_i}$ & 40.31390 & 0.03392\\ [0.2cm]
	 \hline
	 4 & $\frac{r_j}{p_j}\times \frac{\sum_{i=1}^n{p_i}}{\sum_{i=1}^n{r_i}}$  &  1.89904 & 0.029879 & 18 & $Decile(r_j)$& 0.84073 &  2.86997\\ [0.2cm]
	 \hline
	 5 & $\frac{p_j}{r_j}\times \frac{\sum_{i=1}^n{r_i}}{\sum_{i=1}^n{p_i}}$ &  -21.23830 & 0.04374 & 19 & $\frac{r_j}{Decile(r_j)}$& 206.65600 & 0.00482 \\ [0.2cm]
	 \hline
	 6 & $\frac{r_j}{\sum_{i=1}^n{r_i}}$ &  5440.67000 & 0.00839
	 & 20 & $Decile(p_j)$& 0.12019 & 2.86839\\ [0.2cm]
	 \hline
	 7 & $\frac{p_j}{\sum_{i=1}^n{r_i}}$ &  6467.43000 & 0.00077
	 & 21 & $\frac{p_j}{Decile(p_j)}$& 87.99630 &  0.00473\\ [0.2cm]
	 \hline
	 8 & $\frac{r_j+p_j}{\sum_{i=1}^n{r_i}}$ &  5286.99000 & 0.00856 & 22 & $\frac{\#_j}{\#_T}$ & -31.23980 & 0.03438\\ [0.2cm]
	 \hline
	 9 & $\frac{r_j}{\sum_{i=1}^n{p_i}}$ &  47.20590 & 0.59882
	 & 23 & $\frac{[j]^{SRPT}}{n}$& 125.59400 & 0.28865\\ [0.2cm]
	 \hline
	 10 & $\frac{p_j}{\sum_{i=1}^n{p_i}}$ &  -345.58400 & 0.00835	 & 24 & $\frac{|BS^p_j|}{\sum_{i=1}^n |BS^p_i|}$ & 162.31700 & 0.01218\\ [0.2cm]
	 \hline
	 11 & $\frac{r_j+p_j}{\sum_{i=1}^n{p_i}}$ &  47.14900 & 0.59890	 & 25 & $\frac{|BS^r_j|}{\sum_{i=1}^n |BS^r_i|}$& 429.34900 &  0.00878\\ [0.2cm]
	 \hline
	 12 & $\frac{r_j}{\sum_{i=1}^n{r_i+p_i}}$ &  -6733.09000 & 0.00802 & 26 & $\frac{|BG^p_j|}{\sum_{i=1}^n |BG^p_i|}$& 34.17510 &  0.01473\\ [0.2cm]
	 \hline
	 13 & $\frac{p_j}{\sum_{i=1}^n{r_i+p_i}}$ &  -3879.31000 & 0.00066 & 27 & $\frac{|BG^r_j|}{\sum_{i=1}^n |BG^r_i|}$& 39.68920 &  0.03044\\ [0.2cm]
	 \hline
	 14 & $\frac{r_j+p_j}{\sum_{i=1}^n{r_i+p_i}}$ & -6555.45000 &  0.00814
	 & & & & \\ [0.2cm]
	 \hline
	\end{tabular}}\\
    \caption{List of features of a given job $j$, and their value $\theta_k$ in the learned $\bftheta$.}
    \label{tab:TabFeat}
   \end{center}
\end{table}

\subsection{Learning algorithm}
A training database has been built by randomly generating instances $\Gammah$ with $n\in\{50,70,90,110\}$ and $N=100$, thus leading to $4000$ entries in the database. Each entry is defined by:
\begin{center}
	$(n, \sum_j C_j^*, r_1, p_1, \bfphi(1;\Gammah),..., r_n, p_n, \bfphi(n;\Gammah))$,
\end{center}
with $\sum_j C_j^*$ the optimal solution value for $\Gammah$. After learning, we obtain the normalized $\bftheta$ vector, as well as their standard deviations ${\bf \sigma}$, given in Table \ref{tab:TabFeat}. Generating the database took 2 hours of computing time.
We then build the  the sample average approximation of the learning problem of Section~\ref{ssub:perturbed_model_learning_with_fenchel_young_losses} by drawing $100$ samples of $\bfZ$.
Evaluating the objective of the learning problem thus requires to solve $4000 \times 100$ instances of the easy problem with the SPT rule.
We solve the learning problem with the BFGS implementation of~\cite{wieschollek2016cppoptimizationlibrary}.
The BFGS algorithms converges after $91$ iterations, which took a total of 1 hour and 31 minutes.
 


\subsection{Comparisons of the heuristics}
\label{sec:comparisons}

\paragraph{Comparison to optimal solutions}~\\
We first focus on the comparison of the heuristics with the optimal solution on instances $\Gammah$ with $n\in\{50,60, 70,80,90,100,110\}$ and $N=30$. None of the randomly generated instances corresponds to some instances generated when building the training database.
Table \ref{tab:TabCompOpt} presents the obtained results. For each problem size $n$ and each heuristic $H\in\{\texttt{RAND}, \texttt{PMLH}, \texttt{IMLH}, \texttt{itMLH}, \texttt{RDIA}, \texttt{RBS}, \texttt{MATH}\}$, we compute several statistics. Column $\delta_{avg}$ (resp. $\delta_{max}$) is the average (resp. maximum) deviation of heuristic $H$ to the optimal solution of the $N$ instances of size $n$. For a given instance $\Gammah$, the deviation $\delta$ is computed as follows:
\begin{center}
	$\delta=100.00\times \frac{\sum_j C_j(H)-\sum_j C_j^*}{\sum_j C_j^*}$,
\end{center}
with $\sum_j C_j(H)$ the value of the solution returned by $H$ and $\sum_j C_j^*$ the optimal solution value. Column $\# Opt$ is the percentage of instances for which heuristic $H$ finds an optimal solution. Column $T_{avg}$ (resp. $T_{max}$) is the average (resp. maximum) CPU time in seconds used by heuristic $H$.\\
First, the results in Table \ref{tab:TabCompOpt} show that heuristic \texttt{MATH} strongly outperforms all other heuristics but at the price of a CPU time significantly higher. Heuristic \texttt{RBS} is slightly worse than \texttt{MATH} and with a running time signicantly higher than for the other heuristics. These results are in line with those given in \cite{DST14}. 
The poor performance of \texttt{RAND} shows that the quality of the prediction made by $\varphi_{\bftheta}$ is crucial for the performance of \texttt{IMLH}: the local search alone does not enable to find a good solution.
Furthermore, even \texttt{PMLH}, which does not perform a local optimization after the prediction gives interesting results in terms of deviations to optimality. Besides, heuristic \texttt{IMLH} slightly outperforms heuristic \texttt{RDIA} both in terms of deviations and percentage of instances solved to optimality. Remind that the only difference between these two heuristics is the initial solution given to the RDI local search: \texttt{IMLH} uses a predictor and a fast local search, while \texttt{RDIA} uses a dedicated heuristic algorithm. Finally, we can remark that heuristic \texttt{itMLH}, that improves upon \texttt{IMLH}, provides deviations to optimality slightly worse than those of \texttt{RBS} but in a reduced running time.\\
We can conclude from these experiments that:
\begin{enumerate}
	\item Learning based heuristics are competitive with the state-of-the-art heuristics \texttt{RDIA} and \texttt{RBS}.
	\item Learning based heuristics require a reduced running time with respect to \texttt{RBS} and \texttt{MATH}.
	\item Heuristic \texttt{MATH} is a way far the most effective heuristic in terms of deviations to optimality.
\end{enumerate}

\begin{sidewaystable}
	\centering
	{\tiny
		\begin{tabular}{|c|c|c|c|c|c|c|c|c|c|c|c|c|c|c|c|}
			\hline
		 & \multicolumn{5}{|c|}{\texttt{RAND}} & \multicolumn{5}{|c|}{\texttt{PMLH}} &  \multicolumn{5}{|c|}{\texttt{IMLH}}\\
		 \cline{2-16}
		 $n$ & $\delta_{avg} (\%)$ & $\delta_{max} (\%)$ & $\# Opt (\%)$ & $T_{avg}$ (s) & $T_{max}$ (s) & $\delta_{avg} (\%)$ & $\delta_{max} (\%)$ & $\# Opt (\%)$ & $T_{avg}$ (s) & $T_{max}$ (s) & $\delta_{avg} (\%)$ & $\delta_{max} (\%)$ & $\# Opt (\%)$ & $T_{avg}$ (s) & $T_{max}$ (s) \\
		 \hline
         $50$ & 77.779 & 134.602 & 0.00 & 0.00 & 0.00 & 1.491 & 6.643 & 0.00 & 0.00 & 0.00 & 0.208 & 2.371 & 18.33 & 0.01 & 1.00 \\
         \hline
         $60$ & 80.699 & 150.660 & 0.00 & 0.00 & 0.00 & 1.212 & 5.231 & 0.00 & 0.01 & 1.0 & 0.181 & 3.720 & 18.00 & 0.01 & 1.00\\
		 \hline
         $70$ & 84.460 & 139.698 & 0.00 & 0.00 & 0.00 & 1.066 & 5.533 & 0.00 & 0.00 & 0.00 & 0.171 & 1.784 & 11.67 & 0.00 & 1.00 \\     
		 \hline
         $80$ & 86.597 & 147.243 & 0.00 & 0.01 & 1.00 & 0.994 & 4.857 & 0.00 & 0.00  & 1.0 & 0.157 & 2.647 & 7.00 & 0.01 & 1.00\\
		 \hline
         $90$ & 89.537 & 158.000 & 0.00 & 0.00 & 0.00 & 0.973 & 4.414 & 0.00 & 0.00 & 0.00 & 0.128& 1.481 & 10.67 & 0.01 & 1.00\\
		 \hline
         $100$ & 90.700 & 144.174 & 0.00 & 0.00 & 1.00 & 0.919 & 4.276 & 0.00 & 0.01 & 1.00 & 0.118 & 1.578 & 6.67 & 0.01 & 1.00\\
		 \hline
         $110$ & 91.666 & 145.854 & 0.00 & 0.00 & 0.00 & 0.903 & 3.848 & 0.00 & 0.01 & 1.00 & 0.103 & 1.209 & 8.33 & 0.01 & 1.00 \\
		 \hline
		\end{tabular} \\
	     ~\\
	     ~\\
		\begin{tabular}{|c|c|c|c|c|c|c|c|c|c|c|c|c|c|c|c|}
		\hline
		& \multicolumn{5}{|c|}{\texttt{itMLH}} & \multicolumn{5}{|c|}{\texttt{RDIA}} &  \multicolumn{5}{|c|}{\texttt{RBS}}\\
		\cline{2-16}
		$n$ & $\delta_{avg} (\%)$ & $\delta_{max} (\%)$ & $\# Opt (\%)$ & $T_{avg}$ (s) & $T_{max}$ (s) & $\delta_{avg} (\%)$ & $\delta_{max} (\%)$ & $\# Opt (\%)$ & $T_{avg}$ (s) & $T_{max}$ (s) & $\delta_{avg} (\%)$ & $\delta_{max} (\%)$ & $\# Opt (\%)$ & $T_{avg}$ (s) & $T_{max}$ (s) \\
		\hline
		$50$ & 0.055 & 1.157 & 26.00 & 0.51 & 1.00  & 0.229 & 2.687 & 19.00 & 0.00 & 0.00 & 0.014 & 0.302 & 67.00 & 0.43 & 1.00\\
		\hline
		$60$ & 0.048 & 0.571 & 23.33 & 0.55 & 1.00 & 0.191 & 2.058 & 13.00 & 0.00 & 0.00 & 0.017 & 0.426 & 62.00 & 0.73 & 1.00\\
		\hline
		$70$ & 0.048 & 1.079 & 18.33 & 0.67 & 1.00 & 0.218 & 1.804 & 12.67 & 0.00 & 0.00 & 0.015 & 0.158 & 57.33 & 1.31 & 3.00 \\
		\hline
		$80$ & 0.052 & 1.675 & 10.00 & 0.67 & 2.00 & 0.209 & 2.018 & 8.67 & 0.00 & 1.00  & 0.016 & 0.193 & 46.67 & 1.95 & 3.00\\
		\hline
		$90$ & 0.043 & 0.342 & 11.67 & 0.71 & 2.00 & 0.204 & 1.757 & 6.67 & 0.00 & 0.00 & 0.017 & 0.388 & 44.33 & 2.80 & 5.00 \\
		\hline
		$100$ & 0.037 & 0.359 & 8.33 & 0.80 & 2.00  & 0.182 & 2.273 & 7.33 & 0.00 & 0.00 & 0.013 & 0.162 & 43.33 & 3.96 & 7.00\\
		\hline
		$110$ & 0.034 & 0.280 & 8.67 & 0.90 & 2.00 & 0.181 & 1.760 & 6.33 & 0.00 & 1.00 & 0.013 & 0.120 & 37.67 & 5.54 & 7.00\\
		\hline
		\end{tabular} \\
		 ~\\
		~\\
		\begin{tabular}{|c|c|c|c|c|c|}
			\hline
			& \multicolumn{5}{|c|}{\texttt{MATH}} \\
			\cline{2-6}
			$n$ & $\delta_{avg} (\%)$ & $\delta_{max} (\%)$ & $\# Opt (\%)$ & $T_{avg}$ (s) & $T_{max}$ (s)  \\
			\hline
			$50$ & 0.000 & 0.042 & 98.33 & 4.82 & 30.00 \\
			\hline
			$60$ & 0.002 & 0.149 & 98.00 & 7.44 & 34.00 \\
			\hline
			$70$ & 0.001 & 0.067 & 95.67 & 10.53 & 46.00 \\
			\hline
			$80$ & 0.003 & 0.151 & 92.67 & 11.53 & 59.00 \\
			\hline
			$90$ & 0.004 & 0.377 & 90.33 & 14.00 & 81.00 \\
			\hline
			$100$ & 0.001 & 0.038 & 92.00 & 16.88 &  67.00 \\
			\hline
			$110$ & 0.001 & 0.031 & 90.33 & 20.85 & 68.00 \\
			\hline
		\end{tabular} \\
	}
	\caption{Comparison of the heuristics with optimal solutions}
     \label{tab:TabCompOpt}
\end{sidewaystable}

\paragraph{Comparison to the best known solutions}~\\
We now focus on the comparison of the heuristics with the best known solution on larger instances $\Gammah$ with $n\in\{120, 140, 160, 180, 200, 300, 500, 1000, 1500, 2000,$ $2500\}$ and $N=30$. For a given instance $\Gammah$, the best known solution is the best solution found by the heuristics. None of the randomly generated instances corresponds to some instances generated when building the training database. Also notice that a time limit of $180$s is imposed in the experiments: as soon as for a given size $n$, the average running time of a heuristic exceeds $180$s, this one is no longer run for higher values of $n$. The heuristic \texttt{RAND} is not included in these experiments.

Table \ref{tab:TabCompBNS} presents the obtained results. The meaning of the columns is the same than for Table \ref{tab:TabCompOpt}, with the modification that, for a given instance $\Gammah$, the deviation $\delta$ is computed as follows:
\begin{center}
	$\delta=100.00\times \frac{\sum_j C_j(H)-\sum_j C_j^{BNS}}{\sum_j C_j^{BNS}}$,
\end{center}
with $\sum_j C_j(H)$ the value of the solution returned by $H$ and $\sum_j C_j^{BNS}=\min_{H \text{is running}} (\sum_j C_j(H))$ the value of the best known solution.\\
Table \ref{tab:TabCompBNS} shows that heuristics \texttt{MATH} and \texttt{RBS} are unable to solve instances with more than $n=300$ jobs within the time limit of $180$s. Heuristic \texttt{RDIA} is able to solve instances with up to $2000$ jobs but is slower than \texttt{IMLH} and \texttt{itMLH}. Regarding the deviations to the best known solutions, heuristic \texttt{MATH} remains the most effective one, with an average deviation almost equal to 0. Heuristic \texttt{RBS} is the second most effective heuristic but, again, a slow one not able to solve large instances. Heuristics \texttt{IMLH} and \texttt{itMLH} provide very good results in terms of deviations with reduced average running times. \\
We can conclude from these experiments that:
\begin{enumerate}
	\item Learning based heuristics offer a very good trade-off between the quality of the computed solution and the running time required. They also show very low deviations to the best known solution.
	\item Heuristic \texttt{RDIA} is outperformed by both \texttt{IMLH} and \texttt{itMLH}.
	\item Heuristics \texttt{MATH} and \texttt{RBS} are enable to solve instances with more than 300 jobs within the time limit of $180$s.
\end{enumerate}

\begin{sidewaystable}
	\centering
	{\tiny
		\begin{tabular}{|c|c|c|c|c|c|c|c|c|c|c|c|c|}
			\hline
			& \multicolumn{4}{|c|}{\texttt{PMLH}} & \multicolumn{4}{|c|}{\texttt{IMLH}} &  \multicolumn{4}{|c|}{\texttt{itMLH}}\\
			\cline{2-13}
			$n$ & $\delta_{avg} (\%)$ & $\delta_{max} (\%)$ &  $T_{avg}$ (s) & $T_{max}$ (s) & $\delta_{avg} (\%)$ & $\delta_{max} (\%)$ & $T_{avg}$ (s) & $T_{max}$ (s) & $\delta_{avg} (\%)$ & $\delta_{max} (\%)$ & $T_{avg}$ (s) & $T_{max}$ (s) \\
			\hline
			$120$ & 0.927 & 3.610 & 0.02 & 1.00 & 0.106 & 1.279 & 0.03 & 1.00 & 0.034 & 0.648 & 0.97 & 2.00 \\
			\hline
			$140$ & 0.886 & 3.780 & 0.01 & 1.00 & 0.076 & 1.389 & 0.02 & 1.00 & 0.032 & 0.244 & 1.19 & 2.00 \\
			\hline
			$160$ & 0.930 & 3.691 & 0.02 & 1.00 & 0.079 & 0.740 & 0.04 & 1.00 & 0.034 & 0.276 & 1.26 & 3.00 \\     
			\hline
			$180$ & 0.974 & 4.804 & 0.01 & 1.00 & 0.081 & 0.665 & 0.02 & 1.00 & 0.029 & 0.146 & 1.42 & 2.00 \\
			\hline
			$200$ & 0.981 & 4.360 & 0.01 & 1.00 & 0.078 & 1.326 & 0.05 & 1.00 & 0.025 & 0.136 & 1.68 & 3.00 \\
			\hline
			$300$ & 1.146 & 5.821 & 0.02 & 1.00 & 0.060 & 0.750 & 0.05 & 1.00 & 0.025 & 0.193 & 2.37 & 5.00 \\
			\hline
			$500$ & 1.275 & 4.608 & 0.04 & 1.00 & 0.029 & 0.390 & 0.21 & 1.00 & 0.005 & 0.091 & 3.82 & 7.00 \\
			\hline
			$1000$ & 1.453 & 4.885 & 0.07 & 1.00 & 0.024 & 0.234 & 2.04 & 9.00 & 0.007 & 0.069 & 9.85 & 22.00 \\
			\hline
			$1500$ & 1.520 & 4.749 & 0.13 & 1.00 & 0.020 & 0.190 & 11.35 & 41.00 & 0.008 & 0.068 & 26.40 & 84.00 \\
			\hline
			$2000$ & 1.555 & 5.069 & 0.13 & 1.00 & 0.017 & 0.116 & 40.37 & 146.00 & 0.009 & 0.070 & 71.43 & 300.00 \\
			\hline
			$2500$ & 1.559 & 4.668 & 0.16 & 1.00 & 0.008 & 0.135 & 108.68 & 472.00 & 0.000 & 0.000 & 185.44 & 907.00 \\
			\hline
		\end{tabular} \\
		~\\
		~\\
		\begin{tabular}{|c|c|c|c|c|c|c|c|c|c|c|c|c|}
			\hline
			& \multicolumn{4}{|c|}{\texttt{RDIA}} & \multicolumn{4}{|c|}{\texttt{RBS}} &  \multicolumn{4}{|c|}{\texttt{MATH}}\\
			\cline{2-13}
			$n$ & $\delta_{avg} (\%)$ & $\delta_{max} (\%)$ &  $T_{avg}$ (s) & $T_{max}$ (s) & $\delta_{avg} (\%)$ & $\delta_{max} (\%)$ &  $T_{avg}$ (s) & $T_{max}$ (s) & $\delta_{avg} (\%)$ & $\delta_{max} (\%)$ & $T_{avg}$ (s) & $T_{max}$ (s) \\
			\hline
			$120$ & 0.167 & 1.759 & 0.02 & 1.00 & 0.010 & 0.091 & 7.57 & 11.00 & 0.000 & 0.000 & 21.17 & 96.00 \\
			\hline
			$140$ & 0.162 & 1.585 & 0.01 & 1.00 & 0.010 & 0.110 & 13.09 & 17.00 & 0.000 & 0.088 & 29.90 & 122.00 \\
			\hline
			$160$ & 0.171 & 1.421 & 0.02 & 1.00 & 0.010 & 0.066 & 20.38 & 26.00 & 0.000 & 0.000 & 38.97 & 205.00  \\     
			\hline
			$180$ & 0.132 & 0.955 & 0.03 & 1.00 & 0.009 & 0.065 & 31.25 & 38.00 & 0.000 & 0.000 & 49.85 & 151.00 \\
			\hline
			$200$ & 0.138 & 1.172 & 0.08 & 1.00 & 0.008 & 0.055 & 55.70 & 90.00 & 0.000 & 0.000 & 92.54 & 296.00 \\
			\hline
			$300$ & 0.101 & 1.185 & 0.25 & 1.00 & 0.005 & 0.032 & 195.52 & 238.00 & 0.001 & 0.158 & 214.79 & 301.00\\
			\hline
			$500$ & 0.046 & 1.090 & 1.53 & 9.00 & ---- & ---- &  ----& ----& ----& ----& ----& ----\\
			\hline
			$1000$ & 0.036 & 0.788 & 20.76 & 75.00 & ---- & ---- &  ----& ----& ----& ----& ----& ----\\
			\hline
			$1500$ & 0.022 & 0.754 & 94.89 & 496.00 & ---- & ---- &  ----& ----& ----& ----& ---- & ----\\
			\hline
			$2000$ & 0.022 & 0.717 & 314.98 & 917.00 & ---- & ---- &  ----& ----& ----& ----& ----& ----\\
			\hline
			$2500$ & ---- & ----& ----& ----& ----& ----& ----& ----& ----& ----& ----& ----\\
			\hline
		\end{tabular} \\
	}
	\caption{Comparison of the heuristics with best known solutions}
	\label{tab:TabCompBNS}
\end{sidewaystable}


%% file: conclu.tex

\section{Conclusion}
\label{sec:conclusions}

In summary, we introduce several new heuristics for the scheduling problem $1|r_j|\sum_j C_j$.
These heuristics all rely on a structured learning predictor used to approximate an instance of $1|r_j|\sum_j C_j$ by an instance of $1||\sum_j C_j$.
Our heuristics are competitive with RDI-APRTF, the state-of-the-art ``fast'' heuristic for $1|r_j|\sum_j C_j$ on instances with up to 500 jobs, and outperform it on larger instances.
To the best of our knowledge, they are the first machine learning based algorithms to outperform state-of-the-art algorithms on a scheduling problem.
And research directions include the extension of our approach to other scheduling problems.

More generally, our work is a proof of concept for the ML for OR paradigm of \citep{parmentierLearningApproximateIndustrial2019}:
our results on $1|r_j|\sum_j C_j$ show that algorithms obtained using this ML for OR paradigm can outperform state-of-the-art algorithms on classic problems of the OR literature -- the state-of-the-art algorithms for the applications considered in \citep{parmentierLearningApproximateIndustrial2019} were not as challenging.
Furthermore, our results on $1|r_j|\sum_j C_j$ show the relevance of three extensions of the paradigm of \citep{parmentierLearningApproximateIndustrial2019} introduced in this paper.
First, the Fenchel-Young loss approach we introduce for the learning problem leads to practically efficient algorithms.
This paves the way to new applications of the paradigm because this learning method is generic.
And, we can boost the performance of the algorithms obtained with the ML for OR paradigm by using local search heuristics as decoding algorithms as in \texttt{IMLH}, and by perturbating $\bftheta$ as in $\texttt{itMLH}$.
These new ideas are not specific to the $1|r_j|\sum_j C_j$ problem, and future directions include their applications to other operations research problems.

\medskip
\noindent\textbf{Acknowledgments}.
We are grateful to Federico Della Croce for his advice on the algorithms to use in our benchmark, and to Fabio Salassa for providing us the code of the matheuristic.

%% file: app_proofs.tex

\section{Proofs}
\label{sec:proofs}

\begin{proof}[Proof of Proposition~\ref{prop:propertiesOfF}]
    The main statement and the case where $Z$ has positive and differentiable density $\rmd\mu(z) \propto \exp(-\nu(z))\rmd z$ summarize results in Sections~2 and~3 of \citet{berthetLearningDifferentiablePerturbed2020}.
    The sampled minimization case immediately follows from the fact that $F\cdot,x)$ as a maximum of $m$ linear mappings.
\end{proof}

\begin{proposition}[Proof of Proposition~\ref{prop:propertiesOfFYL}]
    The convexity of $L$ is an immediate corollary  of the convexity of $F$ established in Proposition~\ref{prop:propertiesOfF}. 
    The Fenchel-Young inequality implies that the Fenchel-Young loss is non-negative, and equal to $0$ only if $\bfphi(y;x)$ is in the subdifferential of $F$.
    Since, a sum of convex functions among which at least one is strictly convex is strictly convex, the first point follows from the results in Section~4 of \citep{berthetLearningDifferentiablePerturbed2020}[Section 4].
    The second point immediately follows from the fact that $L(\bftheta;y,x) = 0$ if and only if $\bfphi(y;x) \in \partial_{\bftheta}F(\bftheta;x)$ and the fact that $\delta_{\bftheta}\max_{\bfeta \in \calC(x)}\langle \bftheta + \bfz |\bfeta\rangle$ is the set of $\bfeta$ realizing the maximum.
\end{proposition}

%% file: app_noise.tex

\section{Influence of perturbation strength}
\label{app:perturbationStrength}

In this appendix, we explain why the strength of the perturbation does not impact the result of our method.
We therefore suppose to use the perturbation $\bftheta + \varepsilon \bfZ$ wherever we used the perturbation $\bftheta + \bfZ$ in the paper, with $\varepsilon > 0$.
Remark that this latter assumption is w.l.o.g.: Since $\bfZ \sim \calN(0,I_d)$ is symmetric, moving from $\varepsilon$ to $-\varepsilon$ does not change the distribution of $\bftheta + \varepsilon \bfZ$.
Let 
\begin{align*}	
	F_{\varepsilon}(\bftheta;x) &:= \bbE_{\bfZ} \Big( \max_{y \in \calY(x)} \langle \bftheta + \varepsilon \bfZ| \bfphi(y;x)\rangle \Big) = ,\bbE_{\bfZ} \Big( \max_{\bfeta \in \calC(x)} \langle \bftheta + \varepsilon \bfZ| \bfeta\rangle \Big) 
	\quad \text{and,} \\ 
	\bfphi_{\varepsilon}^*(\bftheta;\bfx) &:= \bbE_{\bfZ}\Big(\bfphi\Big(\argmax_{y \in \calY(x)}\big\langle\bftheta + \varepsilon \bfZ | \bfphi(y;x)\big\rangle;x\Big)\Big) = \bbE\Big(\argmax_{\bfeta \in \calC(x)}\langle \bftheta + \varepsilon Z | \bfeta \rangle\Big).
\end{align*}
Let $F^*_{\varepsilon}(\cdot;x)$ be the Fenchel dual of $F_{\varepsilon}(\cdot;x)$ and recall that $\Omega := F_1^*$, i.e., 
$$F^*_{\varepsilon}(\bfeta;x) := \argmax_{\bftheta \in \calC(x)}\langle \bfeta | \bftheta \rangle - F_{\varepsilon}(\bfeta;y) \quad \text{and} \quad \Omega(\bfeta;x) := F_1^*(\bfeta;x).$$
Proposition 2.2 of \citet{berthetLearningDifferentiablePerturbed2020} establishes that  

\begin{equation} \label{eq:influenceOfEpsilonOnF}
    F_{\varepsilon}(\bftheta;x) = \varepsilon F_1(\frac{\bftheta}{\varepsilon}), 
    \quad
    F^*_{\varepsilon}(\cdot;x) = \varepsilon \Omega(\cdot;x),
    \quad \text{and}\quad
    \bfphi_{\varepsilon}^*(\bftheta;\bfx) = \bfphi_{1}^*(\frac{\bftheta}{\varepsilon};\bfx).
\end{equation}
Replacing $\bftheta + \bfZ$ by $\bftheta + \varepsilon \bfZ$  amounts to replace $F$ and $\Omega$ by $F_{\varepsilon}$ and $F_{\varepsilon}^*$ in the learning Problem~\eqref{eq:learningProblem}.
If we denote by $\bftheta^*$ the optimal solution of Problem~\eqref{eq:learningProblem} with the standard perturbation $\varepsilon = 1$, and  $\bftheta_{\varepsilon}^*$ the optimal solution with the perturbation with strength $\varepsilon$, it follows from Equation~\eqref{eq:influenceOfEpsilonOnF} that  $\bftheta^*_\varepsilon = \varepsilon \bftheta^*$.
Hence, the structured prediction problem we obtain for any scenario $\omega$ is 
$$ \argmax_{y \in \calY(x)}\langle \bftheta_{\varepsilon}^* + \varepsilon Z(\omega) | \bfphi(y;x)\rangle = \argmax_{y \in \calY(x)}\varepsilon\langle \bftheta^* + Z(\omega) | \bfphi(y;x)\rangle$$ 
and we obtain the same predictions as the one we could have obtained using $\varepsilon = 1$.